\documentclass[12pt,a4paper]{amsart}
\usepackage{color}
\usepackage[all,cmtip]{xy}
\usepackage{amsfonts, amsmath, latexsym, amsthm, amssymb}
\usepackage{latexsym}
\usepackage[left=3cm,right=2cm,top=3cm,bottom=2cm]{geometry} %Margens
\usepackage[square,sort,comma,numbers]{natbib}
\newtheorem{Teo}{Theorem}
\newtheorem{Lema}[Teo]{Lemma}
\newtheorem{Prop}[Teo]{Proposition}

\newtheorem{Def}[Teo]{Definition}
\newtheorem{Quest}{Question}

\newcommand{\calm}{{\mathcal M}}

\newcommand{\calb}{{\mathcal B}}
\newcommand{\calr}{{\mathcal R}}

\newcommand{\h}{\mathrm{H}}

\DeclareMathOperator{\Proj}{Proj}

\newcommand{\sing}{\operatorname{Sing}}
\newcommand{\PP}{\mathbb{P}^3}
\newcommand{\lhom}{{\mathcal H}{\it om}}
\newcommand{\lext}{{\mathcal E}{\it xt}}

\DeclareMathOperator{\op3}{\mathcal{O}_{\mathbb{P}^3}}

\usepackage[
	pdfauthor={Charles}
	pdftitle={Report I}
	pdfdisplaydoctitle=true,
	pdfstartview={FitH},
	bookmarks=false,
]{hyperref}
\title{The spectrum  of torsion free sheaves on $\PP$ and applications}

\author{Charles  Almeida}
\address{ICEx - UFMG\\
Departamento de Matem\'atica \\ Av. Ant\^onio Carlos, 6627 \\
30123-970 Belo Horizonte MG, Brazil}
\email{charlesalmeida@ufmg.br}

\begin{document}
%\begin{abstract}
%asasassss \end{abstract}

 \begin{abstract}
We study the spectrum of rank $2$ torsion free sheaves on $\PP$ with aim of producing examples of  distinct irreducible components of the moduli space with the same spectrum answering a question  addressed in \cite{Rao1990} for the case of torsion free sheaves. In order to do so, we provide a full description of the spectrum of the sheaves in the moduli space of semistable rank $2$ torsion free sheaves on $\PP$ with Chern classes $(c_1, c_2,c_3)$ equals to $(-1,2,0)$ and $(0,3,0)$.  
\end{abstract}

\maketitle
\tableofcontents 
 
\section{Introduction}

The question about the number of irreducible components of the moduli space of stable rank $2$ locally free sheaves did attract the attention of many mathematicians during the late 80s and early 90s. During this period, several results were obtained, for moduli spaces with small Chern classes (see \cite{Barth1}, \cite{Banica},  \cite{Hart1}, \cite{Hart2}, \cite{M} for some examples).

 Usually,  the approach to obtain such results is to characterize irreducible components inside the moduli space, and then prove that the components found are the only possible. The second task, that is, the exhaustion of all irreducible components are often the most challenging part of the problem. To deal with this problem Barth and Elencwajg in \cite{BE}  introduced the notion of spectrum of a locally free sheaf, that was generalized to rank $2$ reflexive sheaves by Hartshorne  in \cite{Harshorne-Reflexive}, and later for torsion free sheaves with arbitrary rank on $\mathbb{P}^n$ by Okonek and Spindler  in  \cite{{O-S-1984}}.  The spectrum    assigns to each torsion free sheaf a finite number of integers satisfying some  properties (see Section $2$ for details), and  one of their most important features is that they provide a tool to systematic exhaust all possible families of  sheaves in the moduli space.
 
 In \cite{Rao1990}, Rao addressed the problem of finding two different irreducible components of the moduli space of locally free sheaves with the same spectrum, we will call this the Rao's question, until the end of this note. He presented two different families of locally free sheaves with the same spectrum, but, as he pointed out, he did not prove that these families are irreducible components of the moduli space, and at the best of our knowledge this question remains open. 
 
  In this work, we will focus on  Rao's question for the case of torsion free sheaves.  We  study the notion of spectrum of rank $2$ torsion free sheaves on $\PP$, and then  prove that there are distinct irreducible components of the moduli space  of semistable rank $2$ torsion free sheaves on $\PP$, whose general sheaves have the same spectrum.
  
 Next we outline the content of this work. In Section $2$ we recall  the notion of spectrum for rank $2$ torsion free sheaves and some of its properties obtained by Okonek and Spindler at \cite{O-S-1984}, and we obtain a relation between the spectrum of a rank 2 torsion free sheaf on $\PP$ and the number $s = h^0(\lext^2(E, \mathcal{O}_{\PP}))$. In Section 3,  we study the relation between the singularity type of a torsion free sheaf $E$ (see Definition \ref{typesingularity}) and the number $s = h^0(\lext^2(E, \mathcal{O}_{\PP}))$.   

Finally, in Section 4, we  compute all the possible spectrum for  sheaves in the moduli space of rank $2$ Gieseker semistable torsion free sheaves on $\PP$, denoted by $\calm(c_1, c_2, c_3)$ with Chern classes $(c_1, c_2, c_3)$,  for $(c_1, c_2, c_3) = (-1,2,0)$ and for $(c_1, c_2, c_3) = (-1,3,0)$, and prove that the list of spectrum presented  in Table 1 is exhaustive. With the help of these results we give explicit examples of different irreducible components with the same spectrum in these moduli spaces. 
 
{\bf Acknowledgements.}
The author was supported by the FAPESP grant number 2014/08306-4 and 2016/14376-0, and PNDP Post-doctoral grant at IME-USP. I would like to thank Marcos Jardim for encouraging me to write this note and for the several helpful discussions. 

{\bf Notation and Conventions.} In  this work, $\mathbb{K}$ is an algebraically closed field of characteristic  zero; $\PP := \Proj (\mathbb{K}[x,y,z,w])$. We will not make any distinction between vector bundles and locally free sheaves, and $H^i(F)$ will denote the $i$-th  sheaf cohomology group of a sheaf $F$ on $\PP$, and $h^i(F)$ its dimension.  Let $\calm(c_1, c_2, c_3)$ denote the Gieseker--Maruyama moduli scheme of semistable rank 2 torsion free sheaves on $\PP$ with the first, second and third Chern classes equal to $c_1, c_2$
and $c_3$, respectively. In addition, we also define $\mathcal{R}(c_1, c_2, c_3)$ to be the open subset of $\calm(c_1, c_2,c_3)$ consisting of stable reflexive sheaves, and  $\mathcal{B}(c_1, c_2, c_3)$ to be the open subset of $\calm(c_1, c_2,c_3)$ consisting of stable locally free sheaves. Whenever we deal with rank $r$ normalized sheaves, we will denote its first Chern class  by $e$ instead of $c_1$, to emphasize that $e \in \{-r+1, -r+2, \cdots, 0\}$.

\section{The spectrum of torsion free sheaves}

In this section we will present the definition of spectrum given by Okonek and Spindler in  \cite{O-S-1984}, and their main properties. 

The following result, ensures the existence of spectrum for torsion free sheaves.

\begin{Teo}\label{thmdef}
 Let $E$ be a rank $r$ torsion free sheaf on $\PP$, with generic splitting type $(a_1,\cdots,a_r)$, with $a_i \in \mathbb{Z}$, and $a_1 \leq a_2 \leq \cdots \leq a_r$, and $s = h^0(\lext^2(E,\mathcal{O}_{\PP}))$.
 Then there exists a list of $m$ integers $(k_1, k_2, ...,k_m)$, with $k_1 \leq k_2 \leq  \ldots \leq k_m $ such that
 \begin{itemize}
  \item[a)] $h^1(E(l)) = s + \displaystyle \sum_{i = 1}^{m}h^0(\mathcal{O}_{\mathbb{P}^1}(k_i + l + 1)$ if $l \leq -a_s-1$
  \item[b)] $h^2(E(l)) = \displaystyle \sum_{i = 1}^{m}h^1(\mathcal{O}_{\mathbb{P}^1}(k_i + l + 1)$ if $l \geq a_1 - 3$
 \end{itemize}
\end{Teo}

\iffalse
\begin{Teo}
 Let $E$ be a rank $2$ torsion free sheaf on $\PP$, with generic splitting type $(a_1,a_2)$, with $a_i \in \mathbb{Z}$, and $a_1 \leq a_2$, and $s = h^0(\lext^2(E,\mathcal{O}_{\PP}))$.
 Then there exists a list of $m$ integers $(k_1, k_2, ...,k_m)$, with $k_1 \leq k_2 \leq  \ldots \leq k_m $ such that
 \begin{itemize}
  \item[a)] $h^1(E(l)) = s + \displaystyle \sum_{i = 1}^{m}h^0(\mathcal{O}_{\mathbb{P}^1}(k_i + l + 1)$ if $l \leq -a_2-1$
  \item[b] $h^2(E(l)) = \displaystyle \sum_{i = 1}^{m}h^1(\mathcal{O}_{\mathbb{P}^1}(k_i + l + 1)$ if $l \geq a_1 - 3$
 \end{itemize}
\end{Teo}
\fi
\begin{proof}
See \cite[Theorem 2.3]{O-S-1984}.
\end{proof}
\begin{Def}
Let $E$ be a rank $r$ torsion free sheaf on $\PP$, with generic splitting type $(a_1,\cdots,a_r)$, with $a_i \in \mathbb{Z}$, and $a_1 \leq a_2 \leq \cdots \leq a_r$, and $s = h^0(\lext^2(E,\mathcal{O}_{\PP}))$.
 Then the list of $m$ integers $(k_1, k_2, ...,k_m)$, provided by the previous theorem is called {\it spectrum} of the sheaf $E$.
\end{Def}
 Next, we see that for a torsion free sheaf $E$ the number $s = h^0(\lext^2(E, \mathcal{O}_{\PP}))$ can be related with the cohomology of the sheaf $E$.

\begin{Lema}\label{sh1}
Let $E$ be a rank $r$ torsion free sheaf on $\PP$, then for $t<<0$,  $h^0(\lext^2(E, \mathcal{O}_{\PP})) = h^1(E(t))$. 
\end{Lema}
\begin{proof}
See \cite[Lemma 2.1]{O-S-1984}.
\end{proof}

The next result relates the Euler characteristic of the sheaf with its spectrum.
\begin{Prop}\label{sumk} Let $E$ be a rank 2 torsion free sheaf on $\mathbb{P}^3$, with generic splitting type $(a_1,a_2)$, spectrum $(k_1,\ldots,k_m)$ and $s = h^0(\lext^2(E,\mathcal{O}_{\PP}))$. If $a_2 - a_1 \leq 2$ then 
$$\displaystyle \sum_{i=1}^{m}k_i = m(a_2 -1) - \chi(E(-a_2 - 1)) - s.$$
\end{Prop}

\begin{proof}
 See \cite[Proposition 2.7 ii)]{O-S-1984}.
\end{proof}

At this point, it is still not clear how many elements a sheaf $E$ on $\PP$ can have in its spectrum. The next chain of results will prove that, for a semistable sheaf, one has $c_2(E) = m$ where $m$ is the number of elements in the spectrum of $E$.

\begin{Prop}\label{computem}Let $E$ be a torsion free sheaf on $\PP$, with spectrum $(k_1,\ldots,k_m)$, and  $H \subseteq \PP$ a generic hyperplane. Then we have:
\begin{equation}
m = \chi(E_H(-a_2-1)).
\end{equation}
\end{Prop}
\begin{proof}
See \cite[Proposition 2.7]{O-S-1984}
\end{proof}

We will use the above results to deduce some important properties of semistable  rank 2 torsion free sheaves on $\PP$.

\begin{Prop}\label{m=c2} Let $E$ be a normalized stable rank 2 torsion free sheaf on $\PP$,  and $H \subseteq \PP$ a generic hyperplane. Then it follows that

\begin{equation}
m = \chi(E_H(-1)) = c_2(E). 
\end{equation}
\end{Prop}
\begin{proof}
Since $E$ is a normalized stable rank 2 torsion free sheaf on $\PP$, we have that the splitting type of $E$ is either $(-1,0)$, if $c_1(E) = -1$ or $(0,0)$, if $c_1(E) = 0$. By Proposition \ref{computem}, it is enough to prove that $c_2(E) = \chi(E_H(-1))$.

Fix a generic hyperplane $H \subset \PP$. Consider the sequence of restriction

\begin{equation}\label{restriction}
0 \to \mathcal{O}_{\PP}(-1) \to \mathcal{O}_{\PP} \to \mathcal{O}_{H} \to 0.
\end{equation}

Tensoring  this sequence by $E(-1)$, one has that

\begin{equation}\label{tor}
0 \to  E(-2) \to E(-1) \to E|_{H}(-1)\to 0,
\end{equation}

\noindent  Note that $Tor_1(E(-1),\mathcal{O}_{\PP})=0$, since it is a torsion subsheaf of the torsion free sheaf $E(-2)$. Computing the Euler characteristic we have that

\begin{equation}
\chi(E|_{H}(-1)) = \chi(E(-1)) - \chi(E(-2))= c_2(E),
\end{equation}
\noindent which give us our result. 
\end{proof}

Combining Propositions \ref{computem} and \ref{m=c2}, we have the following.

\begin{Prop}\label{3chern}
Let $E$ be a normalized, semistable torsion free sheaf on $\PP$ with $s = h^0(\lext^2(E, \mathcal{O}_{\PP}))$. Then Then we have the following.
\begin{itemize}
    \item[a)] If $c_1(E) = -1$, then $c_3(E) = -2 \sum k_i - c_2(E) - 2s$.
    \item[b)] If $c_1(E) = 0$, then $c_3(E) = -2 \sum k_i - 2s$.
\end{itemize}
\end{Prop}
\begin{proof}
Since $c_1(E) = -1$, recall that by Hirzebruch-Riemann-Roch, the Euler characteristic of $E$ is $\chi(E(t)) = \dfrac{1}{6}(t+1)(t+2)(2t+3) - \dfrac{1}{2}(c_2(E)(2t+3)+c_3(E))$, using this, and the fact that $m=c_2(E)$ in Proposition \ref{sumk} we have that $c_3(E) = -2 \sum k_i - c_2(E) - 2s$ which proves item a). 

The proof of item $b)$ is analogous, just recall that the Euler characteristic of $E$ is $\chi(E(t)) = \dfrac{1}{3}(t+1)(t+2)(t+3) - (c_2(E)(t+2) +\dfrac{1}{2}c_3(E))$.
\end{proof}

The following piece of notation will be used  to obtain  a result that give us numerical properties for the spectrum of a sheaf. For any torsion free sheaf $E$, on $\PP$  there exists an exact sequence of the form

\begin{equation}\label{resolution}
0 \to R \to F \to E \to 0,
\end{equation}

\noindent where $F$ is a locally free sheaf, and $R$ a reflexive sheaf. Applying the functor $\lhom(-, \op3)$ in the sequence \eqref{resolution}, one has that $\lext^2(E,\op3) \simeq \lext^1(R,\op3)$. Hence the support of the sheaf $\lext^2(E,\op3)$ is zero dimensional. Therefore, it is possible to find a hyperplane $H \subset \PP$  such that $H \cap \mathrm{Supp}~~ \lext^2(E,\mathcal{O}_{\PP}) = \emptyset$, with this, we define  $s_{E_H} = h^0(\lext^1(E,\mathcal{O}_{\PP}) \otimes \mathcal{O}_H)$.   
The next proposition, together with Proposition \ref{sumk} will be useful when computing all possible spectrum for sheaves with fixed Chern classes. 
\begin{Prop}\label{property1}
Let $E$ be a rank 2 torsion free sheaf on $\mathbb{P}^3$, with splitting type $(a_1,a_2)$ and spectrum $(k_1,\cdots,k_m)$. 
\begin{itemize}
\item[a)] Let $k > a_2 + 1$, if there is  at least $s_{E_H} + 1$ elements $k_i$ in the spectrum, such that $k_i \geq k$, then each $k^{'}$, such that $a_2 + 1 \leq k^{'} \leq k$ appears in the spectrum.

\item[b)] If $k \leq a_1 - 1$ is in the spectrum, so every integer $k^{'}$, such that $k \leq k^{'} \leq -1$. Appears in the spectrum.
\end{itemize}
\noindent Furthermore,  if $E$ is reflexive then the following  holds:
\begin{itemize}
 \item[c)] $(k_1, \cdots, k_m)$ is symmetric around $0$.
\end{itemize}
\end{Prop}
\begin{proof}
See \cite[Prop 2.4]{O-S-1984} for the items a) and b), and \cite[Proposition 7.2]{Harshorne-Reflexive} for the item c).
\end{proof}

If we consider rank 2 semistable torsion free sheaves on $\PP$, then their  Chern classes determine the splitting type, and by Proposition \ref{m=c2}, it determines the number of possible elements in the spectrum, and hence it determines all possible sequences of integers that can be the spectrum  of some torsion free sheaf. In \cite{HR1991}, the authors investigated this question and proved that for $c_1 = 0$ and $c_2 $ up to $19$ any sequence of integers satisfying the properties of Proposition \ref{property1}, occurs as  the spectrum   of some locally free sheaf on $\PP$ (see \cite[Theorem 2.1]{HR1991}). It is still not known if the same is true for Chern classes greater than $19$, and it would be interesting to determine when a sequence of integers can be the spectrum for some sheaf. We end this section with an upper bound for the number $s= h^0(\lext^2(E, \mathcal{O}_{\PP}))$ when $E$ is a torsion free sheaf with $c_3(E) = 0$.

\begin{Prop}\label{bound}
 Let $E$ be a normalized semistable rank $2$ torsion free sheaf on $\PP$ with Chern classes $c_1(E) = e$, $c_2(E) = c_2$ and $c_3(E) = 0$, with $e \in \{-1,0\}$. Additionally, let $s = h^0(\lext^2(E, \mathcal{O}_{\PP}))$. Then the following holds.
 \begin{itemize}
     \item[a)] If $e = 0$, then $s \leq \frac{c_2^2+ c_2}{2} $.
     \item[b)] If $e = -1$, then $s \leq \frac{c_2^2 +3c_2}{2} $.
 \end{itemize}
\end{Prop}
\begin{proof}
We will prove $a)$, the item $b)$ is analogous. Let  $E$ be a normalized semistable rank $2$ torsion free sheaf on $\PP$ with $c_3(E)=0$ such that $s = h^0(\lext^2(E, \mathcal{O}_{\PP}))$.  By Proposition \ref{3chern}, we have that $s = -\sum_{i=1}^{c_2} k_i$, where $(k_1,...,k_{c_2})$ is the spectrum of $E$. Given a sequence of integers $(n_1,...,n_{c_2})$ satisfying the properties of Proposition \ref{property1}, the maximum value of $-\sum_{i=1}^{c_2} n_i$ is attained when $(n_1,...,n_{c_2}) = (-c_2,-c_2+1,...,-1)$, therefore 

$$s \leq  \sum_{i=1}^{c_2} i  = \dfrac{c_2^2+ c_2}{2}, $$

\noindent from which our result follows. 
\end{proof}

In Section 4 we will see that the above inequality is not sharp for sheaves on $\PP$, with Chern classes $(-1,2,0)$.  
\section{Spectrum and singularity type of a torsion free sheaf}

In this section we will study the relation between the singularities of a torsion free sheaf $E$ and the number $s= h^0(\lext^2(E, \mathcal{O}_{\PP}))$. First we will recall the notion of singularity type of a torsion free sheaves introduced by Jardim, Markushevich and Tikhomirov in \cite[Introduction]{JMT2016}. 

\begin{Def}\label{typesingularity}Let $E$ be a torsion free sheaf on $\PP$, and set $Q_E := E^{\vee \vee} \slash E$, which we assume to be nontrivial;  then we have the following short exact sequence:
\begin{equation}\label{1}
0 \to E \to E^{\vee \vee} \to Q_E \to 0
\end{equation}
and say that $E$ has
\begin{enumerate}
\item \textit{0-dimensional singularities} if $\dim Q_E =0$;
\item \textit{1-dimensional singularities} if $Q_E$ has pure dimension 1;
\item \textit{mixed singularities} if $\dim Q_E =1$, but $Q_E$ is not pure.
\end{enumerate}
\end{Def}

Given a normalized rank $2$ semistable sheaf on $\PP$, with 0-dimensional singularities,  we will prove next that the number $s = h^0(\lext^2(E, \mathcal{O}_{\PP}))$ is bounded by a formula depending only on the Chern classes of the sheaf. This result will help us to characterize all possible spectrum for sheaves with fixed Chern classes. 

\begin{Prop}\label{bound}
 Let $E$ be a normalized semistable rank $2$ torsion free sheaf on $\PP$, with 0-dimensional singularities and Chern classes $c_1(E) = e$, $c_2(E) = c_2$ and $c_3(E) = 0$, with $e \in \{-1,0\}$. Additionally, let $s = h^0(\lext^2(E, \mathcal{O}_{\PP}))$. Then we have the following.
 \begin{itemize}
     \item[a)] If $e = 0$, then $s \leq \dfrac{c_2^2- c_2 +2}{2} $;
     \item[b)] If $e = -1$, then $s \leq \dfrac{c_2^2}{2} $ if $c_2$ is even, and $s \leq \dfrac{c_2^2-1}{2} $ if $c_2$ is odd.
 \end{itemize}
 \noindent Moreover, these bounds are sharp. 
\end{Prop}
\begin{proof}
First we will prove $a)$. Let $E$ be as in the conditions of the theorem, and $e=0$. Assume by contradiction that $s \geq \dfrac{c_2^2- c_2 +2}{2} +1 $. Since $E$ is a torsion free sheaf, the sheaf $Q_E$ in the sequence \eqref{1} is an artinian sheaf of length $s$. It follows that the total Chern polynomial of $Q_E$ is $c_t(Q_E) = 1+2st^3$  hence  $c_2(E^{\vee \vee}) = c_2(E)$,  and $c_3(E^{\vee \vee}) = 2s$. This implies that  $c_3(E^{\vee \vee}) = c_2^2- c_2 +3 (*)$. Since $E^{\vee \vee}$ is reflexive, it follows from $(*)$ and  \cite[Theorem 8.2 a)]{Harshorne-Reflexive} that $E^{\vee \vee}$ is properly semistable,  which is a contradiction according to \cite[Lemma 3]{JMT2016} that says that if  $E$  is a semistable rank 2 torsion free sheaf on $\PP$, with 0-dimensional singularities, and $c_3(E) = 0$, then $E^{\vee \vee}$ must be stable. Therefore we have that $s \leq \dfrac{c_2^2- c_2 +2}{2} $. 

To prove the item $b)$, let $e =-1$. Note that,  $E$ (hence $E^{\vee \vee}$) is stable since the first Chern class of $E$ is odd.  Assume by contradiction that $s \geq \dfrac{c_2^2}{2} +1$, repeating the same argument as before, we get that $c_3(E^{\vee \vee}) \geq c_2^2+2$, but, by  \cite[Theorem 8.2 d)]{Harshorne-Reflexive} it follows that $E^{\vee \vee}$ is properly semistable, which is a contradiction. 

To see that the bounds are sharp, for the item a) consider $F \in  \calr(0,c_2, c_2^2-c_2+2)$, and let $S$ be the union of $\dfrac{c_2^2- c_2 +2}{2}$ distinct points  $p_i \in \PP$, such that $p_i \not \in \sing F$. Then there exists a surjective morphism $\varphi : F \to \mathcal{O}_S$.  Since $p_i \not \in \sing F$, consider an open cover $ \displaystyle \bigcup U_i$ of $\PP$, such that each $U_i$ contains $p_i$ and none of the other points, neither the singularities of $F$.  Hence $F$ trivializes to $2 \cdot \mathcal{O}_{U_i}$ on each $U_i$. Therefore, for each $i$, we have an epimorphism $\varphi_i: 2 \cdot\mathcal{O}_{U_i} \to \mathcal{O}_{p_i,U_i}$ that glues to an epimorphism $\varphi : F \to \mathcal{O}_S$.

It is easy to check that  $E := \ker \varphi$ is a semistable rank 2 torsion free sheaf in $\mathcal{M}(0,c_2,0)$ such that $E^{\vee \vee} \simeq F$, and $Q_E \simeq \mathcal{O}_S$, therefore, $s = h^0(Q_E) = \mathrm{length}~Q_E = \dfrac{c_2^2- c_2 +2}{2}$. The proof for the item b) is analogous, just consider the sheaves $F \in \calr(-1, c_2, c_2^2)$, and $S$  the  union of $\dfrac{c_2^2}{2}$ for $c_2$ even (or $\dfrac{c_2^2-1}{2}$ for $c_2$ odd)  distinct points  $p_i \in \PP$, such that $p_i \not \in \sing F$. \end{proof}

 It is important to remark that there exists semistable rank $2$ torsion free sheaves on $\PP$, for instance, consider  $p_1,p_2 \in \PP$, two distinct points. The sheaf $E = \mathcal{I}_{p_1} \oplus \mathcal{I}_{p_2}$ is  a rank $2$ torsion free sheaf on $\PP$ such that $c_3(E) = -4$, where $\mathcal{I}_{p_i}$ denotes the ideal sheaf of $p_i$, for $i=1,2$. However, the interest in the study of torsion free sheaves with the third Chern class being non-negative, is that the moduli space $\calm(c_1,c_2,c_3)$ is the compactification of $\calr(c_1, c_2, c_3)$, that is, some of the sheaves in $\calm(c_1, c_2, c_3)$ can be seen as deformation of sheaves in $\calr(c_1, c_2, c_3)$. In particular, $\calm(c_1,c_2,0)$ is the compactification of $\calb(c_1, c_2, 0)$,

The next result characterizes when a sheaf $E$ has pure 1-dimensional singularities in terms of the number $s = h^0(\lext^2(E,\mathcal{O}_{\PP}))$.

\begin{Teo} Let $E$ be rank $2$ torsion free sheaf $\PP$. Then, $E$ has pure 1-dimension singularities if, and only if, $s = h^0(\lext^2(E,\mathcal{O}_{\PP})) = 0$. 
\end{Teo}
\begin{proof}

Note that for $t<<0$, we have $s = h^1(E(t))$ by Lemma \ref{sh1}.  By the long exact sequence of cohomology of the sequence (\ref{1}), we see that $h^1(E(t)) = h^0(Q_E(t))$.

%the by local to global spectral of the sheaf $\lext$, we see that $\ext^2(E(t), \mathcal{O}_{\PP}) \simeq H^0(\lext^2(E(t), \mathcal{O}_{\PP}))$

%$h^1(E(t)) = \dim \ext^1(\mathcal{O}_{\PP},E(t)) = \dim \ext^2(E(t), \mathcal{O}_{\PP}(-4))$

%Consider the exact sequence 

%\begin{equation}\label{sequenceE}
%0 \to E \to E^{\vee \vee} \to Q_E \to 0
%\end{equation}

%\noindent applying the functor $\lhom(-,\mathcal{O}_{\PP})$ in the sequence (\ref{sequenceE}) one sees that $\lext^2(E,\mathcal{O}_{\PP}) \simeq \lext^3(Q_E,\mathcal{O}_{\PP})$. Using  local to global spectral sequence  one can see that $\ext^3(Q_E,\mathcal{O}_{\PP}) \simeq H^0(\lext^3(Q_E,\mathcal{O}_{\PP}))$. Therefore by Grothendieck-Serre-Duality we have that $s_E = h^0(Q_E(-4))$. 
\noindent Now, let $Z_E \subset Q_E$ be the maximal $0$-dimensional subsheaf of $Q_E$, and $T_E = Q_E/Z_E$. We have the following exact sequence:

\begin{equation}\label{torsion}
0 \to Z_E \to Q_E \to T_E \to 0.
\end{equation}

Now, if $s = 0$, then using the long exact  sequence in cohomology from (\ref{torsion}), we see that $h^0(Z_E(t)) = 0$ since $h^0(Q_E(t)) = s= 0$ for $t<<0$. Since $Z_E$ has finite length it implies that $Z_E=0$, hence $E$ has pure 1-dimensional singularities. On the other hand, suppose that $E$ has pure 1-dimensional singularities, then  $h^0(Q_E(t)) =0 $  for $t<<0$ since $Q_E$ is supported on a pure 1-dimensional scheme. Once $ h^0(Q_E(t)) = h^1(E(t))$ for $t<<0$, and the former is equal to $s$, we have our result.
\end{proof}

\section{ Applications to moduli spaces}

In this section we will study the behaviour of the spectrum of torsion free sheaves in the different irreducible components of their moduli spaces. We will show that if we allow torsion free sheaves instead of locally free sheaves, then there are examples of different irreducible components with the same spectrum. We will recall some results from \cite{CJT2018} that will be used here.

\begin{Teo}\label{NewComponentsmixed} Given positive integers $n,m$  such that exists an irreducible component  $\calr^{*}(e,n,m) \subset \mathcal{R}(e,n,m)$ with the expected dimension, $8n-3+2e$. For each $r \geq 2$, and $s$ such that $0 \le s \le 2r+2+e-m$, or,  for $r = 1 $, $s = 0$, and $n = m = 1$, there exists  an irreducible component  $\mathrm{X}(e,n,m,r,s) $ of $\mathcal{M}(e,n+1,m+2+e-2r-2s)$ of dimension $8n+4s+2r+2+e$. Such that the general sheaf $E \in \mathrm{X}(e,n,m,r,s)$ fits into an exact sequence of the form
\begin{equation}
    0 \to E \to E^{\vee \vee} \to \mathcal{O}_S \oplus \mathcal{O}_l(r) \to 0,
\end{equation}
\noindent where $E^{\vee \vee}\in \calr^{*}(e,n,m)$, $\mathcal{O}_S = \oplus_{1}^s \mathcal{O}_{p_i}$ where $p_i$ are closed points of $\PP$ outside Sing $E^{\vee \vee}$ and $l$ is a line in $\PP$ not intersecting the points $p_i$ nor Sing $F$. 
\end{Teo} 

\begin{proof}
See \cite[Theorem 10]{CJT2018}.
\end{proof}

\begin{Teo} \label{0dcomp} \normalfont 
For every nonsingular irreducible component $\mathcal{R}^{*}(e,n,m)$ of $\mathcal{R}(e,n,m)$ of expected dimension $8n-3+2e$,  there exists an irreducible component $\mathrm{T}(e,n,m,s)$ of dimension $8n-3+2e+4s$ in $\mathcal{M}(e,n,m-2s)$ whose generic sheaf $E$ satisfies $E^{\vee\vee}\in\mathcal{R}^{*}(e,n,m)$ and $\mathrm{length}(Q_E)=s$.
\end{Teo} 

\begin{proof}
See \cite[Theorem 9]{CJT2018}.
\end{proof}

%is possible to obtain an example of irreducible component of the moduli space of locally free sheaves, containing sheaves with different spectrum (See for instance CITAR EU). 

\subsection{The moduli space $\mathcal{M}(-1,2,0)$} In \cite{{CJT2018}} there is a complete characterization of all possible irreducible components of $\mathcal{M}(-1,2,0)$. More precisely one has the following theorem. 

\begin{Teo}\label{M(-1,2,0)} 
 The   moduli space of rank 2 stable sheaves on $\PP$ with Chern classes $c_1 = -1, c_2 = 2, c_3 = 0$, $\calm(-1,2,0)$, has exactly $4$ irreducible components, namely:
 \begin{itemize}
 \item[a)] The closure of the family of stable rank $2$ locally free sheaves that are obtained by Serre's construction, as extensions of ideal sheaves of two irreducible conics, with dimension 11. This irreducible component will be denoted by $\mathcal{C}(2)$;
 
 \item[b)] The irreducible component $\mathrm{X}(-1,1,1,1,0)$ of dimension $11$, described by Theorem \ref{NewComponentsmixed}, whose generic element is a torsion free sheaf $E$ such that $E^{\vee \vee} \in \calr(-1,1,1)$ and $Q_E$ is supported on a line.
 
 \item[c)] The irreducible component $\mathrm{T}(-1,2,2,1)$ of dimension $15$ described  in Theorem \ref{0dcomp}, whose generic sheaf is a torsion free sheaf $E$ such that $E^{\vee \vee} \in \calr(-1,2,2)$ and $Q_E$ is a  length $1$ sheaf.
 
 \item[d)] The irreducible component $\mathrm{T}(-1,2,4,2)$ of dimension $19$ described by the Theorem \ref{0dcomp}, whose generic sheaf is a torsion free sheaf $E$ such that $E^{\vee \vee} \in \calr(-1,2,4)$ and $Q_E$ is a  length $2$ sheaf, supported at two distinct points. 
 \end{itemize}
\end{Teo}

\begin{proof}
See \cite[Theorem 26]{CJT2018}.
\end{proof}

\begin{Prop} The general spectrum of each component of $\calm(-1,2,0)$ is given by following table. 
\begin{table}[h]\label{M(2)}
\centering
\caption{Irreducible Components of $\mathcal{M}(-1,2,0)$} %\label{my-label}
\begin{tabular}{|c|c|c|}
\hline
\textbf{Component}                                                                     & \textbf{Dimension}  &  \textbf{General Spectrum}               \\ \hline
\textbf{$\mathcal{C}(2)$}                                                    & 11 &  (-1,0)                     \\ \hline
                                                                                                   
\textbf{$\mathrm{X}(-1,1,1,1,0$)}                                                                           & 11                 &  (-1,0)            \\ \hline
\textbf{$\mathrm{T}(-1,2,2,1)$}& 15 & (-1,-1) \\ \hline
\textbf{$\mathrm{T}(-1,2,4,2)$} & 19 & (-2,-1) \\ \hline       \end{tabular}
\end{table}

\noindent Moreover, all the possible spectra for a sheaf $E \in \calm(-1,2,0)$ are in the above table. 
\end{Prop}

\begin{proof}
Let $E$ be the generic sheaf in $\mathcal{C}(2)$. Then, $E$ is obtained by the extension of two disjoint irreducible conics $Y$, via Hartshorne Serre's Correspondence:

$$0 \to \mathcal{O}_{\PP}(-2) \to E \to I_Y(1) \to 0,$$
\noindent where $I_Y$ is the ideal sheaf of $Y$. Twisting the above sequence by $\mathcal{O}_{\PP}(t)$, and using the long exact sequence of cohomology,  we conclude that the cohomology table of $E$ is given by Table \ref{Table 2}.
\begin{table}[h]\caption{$\dim \h^i(E(k))$ for  the generic $E \in $ $\mathcal{C}(2)$}\label{Table 2}
\begin{tabular}{|c|c|c|}
\hline

  $k \backslash i$&1& 2 \\
  
 \hline
  $-1$& 1& 0  \\
  \hline
  $-2$&0& 1  \\
  \hline
  $-3$& 0& 3   \\
  \hline
  $-4$&0& 5   \\

 \hline
\end{tabular}
\end{table}

\noindent   By Proposition \ref{m=c2}, the spectrum of $E$ has the form $(k_1,k_2)$ for some integers $k_1,k_2$. Using this fact, and the Table \ref{Table 2}, in the Theorem \ref{thmdef}, we conclude that the  spectrum of $E$ is $(-1,0)$.

 Let $E$ be the generic sheaf in X$(-1,1,1,1,0)$. By Theorem \ref{NewComponentsmixed}, $E$ fits into an exact sequence of the form:
$$0 \to E \to F \to \mathcal{O}_l(1) \to 0,$$

\noindent where $l \subset \mathbb{P}^3$ is a line and $F \in \mathcal{R}(-1,1,1)$.  By \cite[Remark 1]{CJT2018}, any such $F$ fits into an exact sequence of the form:

$$0 \to \mathcal{O}_{\PP}(-2) \to 3. \mathcal{O}_{\PP}(-1) \to F \to 0. $$

\noindent Twisting the two above sequence by $\mathcal{O}_{\PP}(t)$, and using their long exact sequence of cohomology,  we conclude that the cohomology table of $E$ is given by Table \ref{Table 3}.

\begin{table}[h]\caption{$\dim \h^i(E(k))$ for  the generic $E \in $ X$(-1,1,1,1,0)$ }\label{Table 3}
\begin{tabular}{|c|c|c|}
\hline

  $k \backslash i$&1& 2 \\
  
 \hline
  $-1$& 1& 0  \\
  \hline
  $-2$&0& 1  \\
  \hline
  $-3$& 0& 3   \\
  \hline
  $-4$&0& 5   \\

 \hline
\end{tabular}
\end{table}
 
\noindent Therefore we conclude that the spectrum of $E$ is $(-1,0)$.

Let  $E $ be the generic sheaf in $\mathrm{T}(-1,2,2,1)$, by Theorem \ref{0dcomp}, $E$ fits into and exact sequence of the form

$$0 \to E \to F \to \mathcal{O}_{p} \to 0,$$

\noindent where $F \in \mathcal{R}(-1,2,2)$ and $p \in \PP$, is a closed point such that $p \not \in \mathrm{Sing}~F$. Twisting the  above sequence by $\mathcal{O}_{\PP}(t)$, and using its long exact sequence of cohomology, with help of \cite[Table 2.6.1]{Chang}, where the cohomology table of $F$ is computed  we conclude that the cohomology table of $E$ is given by Table \ref{Table 4}.

\begin{table}[h]\caption{$\dim \h^i(E(k))$ for  the generic $E \in \mathrm{T}(-1,2,2,1)$ }\label{Table 4}
\begin{tabular}{|c|c|c|}
\hline

  $k \backslash i$&1& 2 \\
  
 \hline
  $-1$& 1& 0  \\
  \hline
  $-2$&1& 2  \\
  \hline
  $-3$& 1& 4   \\
  \hline
  $-4$&1& 6   \\

 \hline
\end{tabular}
\end{table}
 
\noindent Therefore we conclude that the spectrum of $E$ is $(-1,-1)$. 

Let  $E $ be the generic sheaf in $\mathrm{T}(-1,2,4,2)$, by Theorem \ref{0dcomp}, $E$ fits into and exact sequence of the form

$$0 \to E \to F \to \mathcal{O}_{p} \oplus \mathcal{O}_{p} \to 0,$$

\noindent where  $p,q \in \PP$, are different closed points such that $p,q \not \in \mathrm{Sing}~F$ and $F \in \mathcal{R}(-1,2,4)$. By \cite[Lemma 9.5]{Harshorne-Reflexive}, any such $F$ fits into an exact sequence of the form:

$$0 \to \mathcal{O}_{\PP}(-1) \to F \to I_Y \to 0, $$

\noindent where $Y$ is a conic in $\PP$.  Twisting the two above sequence by $\mathcal{O}_{\PP}(t)$, and using their long exact sequence of cohomology,  we conclude that the cohomology table of $E$ is given by Table \ref{Table 5}.

\begin{table}[h]\caption{$\dim \h^i(E(k))$ for  the generic $E \in \mathrm{T}(-1,2,4,2)$ }\label{Table 5}
\begin{tabular}{|c|c|c|}
\hline

  $k \backslash i$&1& 2 \\
  
 \hline
  $-1$& 2& 1  \\
  \hline
  $-2$&2& 3  \\
  \hline
  $-3$& 2& 5   \\
  \hline
  $-4$&2& 7   \\

 \hline
\end{tabular}
\end{table}
\noindent Therefore we conclude that the spectrum of $E$ is $(-2,-1)$. 

By Theorem \ref{M(-1,2,0)}, the above considerations exhausts all irreducible components of $\calm(-1,2,0)$. Therefore now we shall show that the list of spectrum appearing in Table 1 exhausts all possibilities. Indeed,  Let $E \in \calm(-1,2,0)$, by Proposition \ref{m=c2}, the spectrum of $E$ has the form $(k_1,k_2)$ for some integers $k_1,k_2$. Since  $E$  is semistable, the splitting type of $E$ is $(-1,0)$. Applying Proposition \ref{property1}, we see that the only possibilities for $k_1$ and $k_2$  are $(-1,0)$, $(-1,-1)$ and $(-2,-1)$ what concludes our proof.  
\end{proof}

In particular, the general sheaves in the components $\mathcal{C}(2)$ and X$(-1,1,1,1,0)$ have the same spectrum $(-1,0)$, thus answering the question proposed by Rao in \cite{Rao1990} for the case of torsion free sheaves. It is important to highlight that, for locally free sheaves, it still not known if there are two irreducible components of the moduli space in which the general sheaves have the same spectrum. Moreover this result shows that the inequality given by Proposition \ref{bound}b) is not sharp. 

\subsection{The moduli space $\mathcal{M}(0,3,0)$  }

The first known irreducible components of  $\mathcal{M}(0,3,0)$ are those corresponding to the locally free sheaves. Clearly $\mathcal{B}(0,3,0) \subset \calm(0,3,0)$, and in \cite{ES} the authors proved that $\mathcal{B}(0,3,0)$ has exactly two irreducible components, namely the Instanton component, whose generic sheaf corresponds to those who can be obtained as the cohomology of a monad of the form:

\begin{equation}\label{instanton}
0 \to 3.\mathcal{O}_{\PP}(-1) \to 8.\mathcal{O}_{\PP} \to 3.\mathcal{O}_{\PP}(1) \to 0,
\end{equation}

\noindent and the Ein component whose  generic sheaf corresponds to those who can be obtained as the cohomology of a monad of the form:

\begin{equation}\label{ein}
0 \to \mathcal{O}_{\PP}(-2) \to \mathcal{O}_{\PP}(-1) \oplus  2. \mathcal{O}_{\PP} \oplus  \mathcal{O}_{\PP}(1) \to \mathcal{O}_{\PP}(2) \to 0.
\end{equation}

In \cite{JMT2016} the authors described the irreducible components  of the moduli space $\calm(0,3,0)$. These components can be also obtained by the Theorem \ref{NewComponentsmixed}. Indeed, is easy to check that $\mathrm{T}(0,3,2,1), \mathrm{T}(0,3,4,2)$,  $\mathrm{T}(0,3,6,3)$ and $\mathrm{T}(0,3,8,4)$ are irreducible components of $\calm(0,3,0)$, with dimension $25,29,33$ and $37$, respectively.  Moreover, they found an irreducible component, denoted by $\mathcal{C}$ whose generic sheaf $E$ has singularities along a smooth plane cubic $C$ and satisfies the following exact sequence:
\begin{equation}\label{C}
0 \to E \to 2.\mathcal{O}_{\PP} \to L(2) \to 0,
\end{equation}
\noindent where $L$ is a line bundle over $C$ with Hilbert polynomial $P_L(k) = 3k$ such that $L \not \simeq \omega_C$, where $\omega_C$ is the canonical sheaf of $C$. 

Additionally, in \cite{IT2017}, the authors obtained $3$ irreducible components whose general sheaves have mixed singularities. Those components can also be obtained by the Theorem \ref{NewComponentsmixed}, and they are identified with $\mathrm{X}(0,2,2,2,0)$, $\mathrm{X}(0,2,4,3,0)$ and $\mathrm{X}(0,2,4,2,1)$, irreducible components of $\mathcal{M}(0,3,0)$ with dimension $22$, $24$ and $26$, respectively. 

With the help of \cite[Table 2.8.1]{Chang}, and \cite[Table 2.12.2]{Chang}, we can compute the general spectrum of the sheaves in the components $\mathrm{X}(0,2,2,2,0)$, $\mathrm{X}(0,2,4,3,0)$ and $\mathrm{X}(0,2,4,2,1)$, and with the help of   \cite[Table 3.4.1]{Chang}, \cite[Table 3.5.1]{Chang}, \cite[Table 3.7.1]{Chang} and \cite[Table 3.9.1]{Chang} we can compute the general spectrum of the sheaves in the components $\mathrm{T}(0,3,2,1), \mathrm{T}(0,3,4,2)$,  $\mathrm{T}(0,3,6,3)$ and $\mathrm{T}(0,3,8,4)$. The general spectrum of the sheaves in the Instanton and Ein components can be computed by the display of the monads \eqref{instanton} and \eqref{ein}. Finally the general spectrum in the component $\mathcal{C}$ can be computed from the cohomology of the sheaf $L$ (that can be obtained by the Hilbert polynomial of $L$ and Riemann-Roch) in the sequence \eqref{C}. Summarizing the above discussion, we have the following proposition. 

\begin{Prop} The Table \ref{Table 6} describes the general spectrum for all the known irreducible components of $\calm(0,3,0)$.
\begin{table}[h]\label{M(3)}
\centering
\caption{Spectra for the Known Irreducible Components of $\mathcal{M}(0,3,0)$}\label{Table 6}
\begin{tabular}{|c|c|c|c|}
\hline
\textbf{Component}                                                                     & \textbf{Dimension}  &  \textbf{General Spectrum}      \\ \hline
\textbf{Instanton}  & 21 &  (0,0,0)  \\ \hline
\textbf{Ein}& 21  &  (-1,0,1) \\ \hline
$\mathcal{C}$& 21 & (0,0,0)    \\ \hline    
\textbf{$\mathrm{X}(0,2,2,2,0)$} & 22 & (-1,0,1) \\  \hline
\textbf{$\mathrm{X}(0,2,4,3,0)$} & 24 & (-1,-1,2) \\  \hline
\textbf{$\mathrm{X}(0,2,4,2,1)$} & 26 & (-1,-1,1) \\  \hline
\textbf{$\mathrm{T}(0, 3, 2, 1)$} & 25 & (-1,0,0) \\  \hline 
\textbf{$\mathrm{T}(0, 3, 4, 2)$} & 29 & (-1,-1,0) \\  \hline  
\textbf{$\mathrm{T}(0, 3, 6, 3)$} & 33 &(-2,-1,0) \\  \hline
\textbf{$\mathrm{T}(0, 3, 8, 4)$} & 37 & (-2,-1,-1) \\  \hline
\end{tabular}
\end{table}
\end{Prop}

Consider now  a sequence of integers $(k_1,k_2,k_3)$  satisfying Proposition \ref{property1} for $m=3$ and $(a_1, a_2) = (0,0)$. It is easy to see that the possibilities for $k_1$, $k_2$ and $k_3$  are $(0,0,0)$, $(-1,0,1)$, $(-1,-1,2)$, $(-1,-1,1)$, $(-1,0,0)$, $(-1,-1,0)$, $(-2,-1,0)$, $(-2,-1,-1)$, $(-2,-2,-1)$ and $(-3,-2, -1)$. Except for  $(-2,-2,-1)$ and $(-3,-2, -1)$, all the other possibilities appear in Table \ref{Table 6}. We have the following question.

\begin{Quest} Are there sheaves in $\calm(0,3,0)$ with spectrum $(-2,-2,-1)$ or $(-3,-2,-1)$? 
\end{Quest}
 We  note  that there exists examples of slope semistable sheaves with these two spectrum. Indeed, consider a rank 2 semistable reflexive shef $F$ on $\PP$, such that $c_1(F) = 0, c_2(F) = 3$ and $c_3(F) = 12$ (see \cite[Example 8.2.1]{Harshorne-Reflexive} to see how such sheaf can be constructed) and $\mathcal{O}_S$ the structure sheaf of  $6$ closed points in $\PP$, lying outside the singularities of $F$. Let $\varphi : F \to \mathcal{O}_S$ be an epimorphism constructed as in the proof of Proposition \ref{bound}, then the sheaf $E := \ker \varphi$ is slope semistable and has spectrum equals to  $(-3,-2,-1)$.  

Now, let $Y$ be a plane quartic of genus $2$. The sheaf $F'$ obtained by the extension $0 \to \mathcal{O}_{\PP} \to F'(1) \to \mathcal{I}_Y(2) \to 0$ is semistable (see \cite[Proposition 4.2]{Harshorne-Reflexive} with Chern classes $c_1(F') = 0, c_2(F') = 3$ and $c_3(F') = 10$. $\mathcal{O}_{S'}$ the structure sheaf of  $5$ closed points in $\PP$, lying outside the singularities of $F'$. Let $\varphi' : F' \to \mathcal{O}_{S'}$ be an epimorphism constructed as in the proof of Proposition \ref{bound}, then the sheaf $E' := \ker \varphi'$ is slope semistable and has spectrum equals to  $(-2,-2,-1)$.

The natural candidates to answer our question are the sheaves $E$ and $E'$ constructed above, however note that $h^0(\lext^2(E,\mathcal{O}_{\PP})=6$ and that $h^0(\lext^2(E,\mathcal{O}_{\PP})=5$, hence by  Proposition \ref{bound} $E$ and $E'$ are not semistable sheaves. 
Clearly the construction given by Theorems  \ref{NewComponentsmixed} and \ref{0dcomp} do not provide such examples either. We believe that the answer to this question could help to compute the exact number of irreducible components for $\calm(0,3,0)$.

\iffalse

 \fi
\bibliography{mref}
\bibliographystyle{abbrv}

\end{document}